\documentclass[twoside,12pt]{amsart}
\usepackage{amsmath}
\usepackage{amsthm}
\usepackage{amsfonts}
\usepackage{latexsym}
\usepackage{float}
\usepackage{upref}
\restylefloat{figure}
\usepackage{mathptmx}
\usepackage{graphicx,enumitem}
\usepackage[thicklines]{cancel}

\makeatletter
\def\serieslogo@{}
\makeatother \makeatletter
\def\@setcopyright{}

\usepackage{amsmath}
\usepackage{amsfonts}
\usepackage{amssymb}
\usepackage{latexsym}
\usepackage{units}
\usepackage{placeins}
\usepackage{bbm}

\newtheorem{theorem}{Theorem}[section]
\newtheorem{algorithm}{Algorithm}[section]
\newtheorem{lemma}[theorem]{Lemma}
\newtheorem{corollary}[theorem]{Corollary}

\theoremstyle{definition}

\newtheorem{remark}{Remark}[section]

\allowdisplaybreaks
\numberwithin{equation}{section}

\newcommand{\cell}{I}  

\renewcommand{\i}{\ifmmode\mathit{\mathchar"7010 }\else\char"10 \fi}
\renewcommand{\j}{\ifmmode\mathit{\mathchar"7011 }\else\char"11 \fi}

\newcommand{\R}{\mathbb{R}}
\newcommand{\N}{\mathbb{N}}

\newcommand{\Ik}{\mathbbm{1}_{\!\cell_k}}
\newcommand{\Z}{\mathbb{Z}}

\newcommand{\sgn}{{\rm sgn}}

\newcommand{\hf}{{\unitfrac{1}{2}}}
\newcommand{\thf}{{\unitfrac{3}{2}}}
\newcommand{\fhf}{{\unitfrac{5}{2}}}
\newcommand{\iphf}{{i+\hf}}
\newcommand{\imhf}{{i-\hf}}

\newcommand{\jphf}{{j+\hf}}

\newcommand{\Del}[2]{D_{[\,#1,#1+#2]}}           
\newcommand{\Delii}[2]{D_{[\,#1,#2]}}       
\newcommand{\del}[2]{d_{[\,#1,#1+#2]}}           
\newcommand{\delii}[2]{d_{[\,#1,#2]}}       
\newcommand{\prodii}[2]{\mathop {\cancel{\prod}}_{#1}^{#2}}
\newcommand{\Cu}[2]{C_{\,#1,#2}}
\newcommand{\cu}[2]{c_{\,#1,#2}} 
\newcommand{\rr}{\ell}  
\newcommand{\rs}{n}
\newcommand{\jumplus}[1]{v^+_{#1}}
\newcommand{\juminus}[1]{v^{-}_{#1}} 
\newcommand{\vplus}{\jumplus{\hf}}
\newcommand{\vminus}{\juminus{\hf}}

\renewcommand{\v}{f}
\newcommand{\V}{F}

\newcommand{\hIi}{|\cell_i|}

\renewcommand{\epsilon}{\varepsilon}
\renewcommand{\phi}{\varphi}

\newcommand{\avg}[1]{\overline{#1}}

\newbox\bokstav
\newdimen\hoyde
\def\bgl{{\hbox{$\left\lbrack\vbox to 8.5pt{}\right.\nOspace$}}}
\def\Bgl{{\hbox{$\left\lbrack\vbox to 11.5pt{}\right.\nOspace$}}}
\def\bggl{{\hbox{$\left\lbrack\vbox to 14.5pt{}\right.\nOspace$}}}
\def\Bggl{{\hbox{$\left\lbrack\vbox to 17.5pt{}\right.\nOspace$}}}
\def\bgr{{\hbox{$\left\rbrack\vbox to 8.5pt{}\right.\nOspace$}}}
\def\Bgr{{\hbox{$\left\rbrack\vbox to 11.5pt{}\right.\nOspace$}}}
\def\bggr{{\hbox{$\left\rbrack\vbox to 14.5pt{}\right.\nOspace$}}}
\def\Bggr{{\hbox{$\left\rbrack\vbox to 17.5pt{}\right.\nOspace$}}}
\def\nOspace{\nulldelimiterspace=0pt \mOth}
\def\mOth{\mathsurround=0pt}
\def\Ljmp{\mathopen{\lbrack\!\lbrack}}
\def\Rjmp{\mathclose{\rbrack\!\rbrack}}
\def\bgLjmp{\mathopen{\bgl\mskip-6mu\bgl}}
\def\bgRjmp{\mathclose{\bgr\mskip-6mu\bgr}}
\def\BgLjmp{\mathopen{\Bgl\!\!\Bgl}}
\def\BgRjmp{\mathclose{\Bgr\!\!\Bgr}}
\def\bggLjmp{\mathopen{\bggl\!\!\bggl}}
\def\bggRjmp{\mathclose{\bggr\!\!\bggr}}
\def\BggLjmp{\mathopen{\Bggl\!\!\Bggl}}
\def\BggRjmp{\mathclose{\Bggr\!\!\Bggr}}
\def\jmp#1{
\setbox\bokstav=\hbox{$ \left. #1\right. $}
\hoyde=\ht\bokstav 
\advance\hoyde by \dp\bokstav
\hbox{$
        \ifinner
                \ifdim\hoyde<10pt
                   \Ljmp #1 \Rjmp%
                \else
                   \ifdim\hoyde <11pt
                      \Ljmp #1 \Rjmp%
                   \else
                      \ifdim\hoyde <14pt
                          \bgLjmp #1 \bgRjmp%
                      \else
                          \ifdim\hoyde <20pt
                             \BgLjmp #1 \BgRjmp%
                          \else
                              \bggLjmp #1 \bggRjmp%
                          \fi
                      \fi
                   \fi
                \fi
        \else
                \ifdim\hoyde<8.5pt
                   \Ljmp #1 \Rjmp%
                \else
                   \ifdim\hoyde <11.5pt
                      \bgLjmp #1 \bgRjmp%
                   \else
                      \ifdim\hoyde <14.5pt
                          \BgLjmp #1 \BgRjmp%
                      \else
                          \ifdim\hoyde <17.5pt
                             \bggLjmp #1 \bggRjmp%
                          \else
                              \BggLjmp #1 \BggRjmp%
                          \fi
                      \fi
                   \fi
            \fi
        \fi
$}
}

\begin{document}

\title[ENO reconstruction and ENO interpolation are stable]{ENO reconstruction and ENO interpolation are stable}

\author[Ulrik S. Fjordholm]{Ulrik S. Fjordholm}\address[Ulrik S.Fjordholm]{\newline Seminar for Applied Mathematics, ETH Z\"urich \newline HG J 48, R\"amistrasse 101, Z\"urich, Switzerland.} \email[]{ulrikf@sam.math.ethz.ch}

\author[Siddhartha Mishra]{Siddhartha Mishra} \address[Siddhartha
Mishra]{\newline Seminar for Applied Mathematics, ETH Z\"urich \newline HG G 57.2, R\"amistrasse 101, Z\"urich, Switzerland.} \email[]{smishra@sam.math.ethz.ch}

\author[Eitan Tadmor]{Eitan Tadmor}\address[Eitan Tadmor]{
\newline Department of Mathematics
\newline Center of Scientific Computation and Mathematical Modeling (CSCAMM) 
\newline Institute for Physical sciences and Technology (IPST)
\newline University of Maryland
\newline MD 20742-4015, USA}
\email[]{tadmor@cscamm.umd.edu}

\thanks{SM thanks Prof. Mike Floater of CMA, Oslo for useful discussions. The research of ET was supported  by grants from National Science Foundation, DMS\#10-08397  and the Office of Naval Research, ONR\#N000140910385.} 

\date{June 3, 2011} 
\maketitle

\begin{abstract}
We prove stability estimates for the ENO reconstruction and ENO interpolation procedures. In particular, we show that the jump of the  reconstructed ENO pointvalues at each cell interface has the same sign as the jump of the underlying cell averages  across that interface. We also prove that the jump of the  reconstructed values  can be upper-bounded in terms of the jump of the underlying cell averages. Similar sign properties hold for the ENO interpolation procedure. These estimates, which are shown to hold for ENO reconstruction and interpolation of arbitrary order of accuracy and on non-uniform meshes, indicate a remarkable rigidity of the piecewise-polynomial ENO procedure.
\end{abstract}

\tableofcontents

\section{Introduction and statement of main results}
The acronym ENO in the title of this paper stands for ``Essentially Non-Oscillatory", and it refers to a \emph{reconstruction} procedure, which generates a piecewise polynomial approximation of a function from a given set of its cell averages. The essence of the ENO procedure, which was introduced by Harten et. al. in \cite{HEOC87}, is its ability to accurately recover discontinuous functions. The starting point is a collection  of cell averages $\{\overline{v}_i\}_{i\in\Z}$ over consecutive intervals $\cell_i = [x_\imhf, x_\iphf)$,
\begin{equation}\label{eq:cellavg}
\overline{v}_i := \frac{1}{\hIi}\int_{\cell_i}v(x) dx,
\end{equation}
from which one can form the piecewise constant approximation of the underlying function $v(x)$,
\[
\mathcal{A}v(x) := \sum_{k} \overline{v}_k \Ik (x), \qquad 
\Ik (x)=\left\{\begin{array}{ll}1 & \text{if } x\in \cell_k,\\
0 & \text{if } x\notin \cell_k.\end{array}\right.
\]
But the averaging operator $\mathcal{A}v(x)$ is limited to first order accuracy, whether $v$ is smooth or not; for example, if $v$ has bounded variation then $\|v-\mathcal{A}v\|_{L^1} = {\mathcal O}(h)$. The purpose of the ENO procedure (abbreviated by $\mathcal{R}$) is to reconstruct a higher order  approximation  of $v(x)$ from its given cell averages,
\begin{equation}\label{eq:eno}
\text{\bf ENO:} \qquad \mathcal{A}v(x) = \sum_k \overline{v}_k \Ik(x) \quad \mapsto \quad \mathcal{R}\mathcal{A}v(x) := \sum_k f_k(x) \Ik(x).
\end{equation}
Here, $\v_k(x)$ are polynomials of degree $p-1$ such that the piecewise-polynomial ENO reconstruction $\mathcal{RA}v(x)$ satisfies the following two essential properties.
\begin{description}
\item[Accuracy] First, it is an approximation of $v(x)$ of order $p$ in the sense that 
\begin{equation}\label{eq:acc}
\mathcal{RA}v(x) = v(x) + {\mathcal O}(h^{p}),
\end{equation}
where $h = \max_i |I_i|$. Typically, the requirement for accuracy is sought whenever $v(\cdot)$ is sufficiently smooth in a neighborhood of $x$. Here, however, (\ref{eq:acc}) is also sought at isolated points of jump discontinuities. Thus, if we let $v(x_{i+\hf}+)$ and $v(x_{i+\hf}-)$ denote the point-values of $v(x)$ at the left and right of the interface at $x_{i+\hf}$, then (\ref{eq:acc}) requires that the corresponding reconstructed point-values, $\juminus{i+\hf}:= \mathcal{RA}v(x_{i+\hf}-)=\v_i(x_{i+\hf})$ and $\jumplus{i+\hf}:=\mathcal{RA}v(x_{i+\hf}+)= \v_{i+1}(x_{i+\hf})$, satisfy
\[
|\juminus{i+\hf}-v(x_{i+\hf}-)| + |\jumplus{i+\hf}- v(x_{i+\hf}+)| = \mathcal{O}(h^{p}).
\]
To address this requirement of accuracy, the $\v_i$'s are constructed from neighboring cell averages $\{\overline{v}_{i+j}\}_{j=k}^{k+p-1}$ for some $k\in \{-p+1,\dots, 0\}$. The key point is to choose an \emph{adaptive}   stencil,
\[
i \mapsto \{\overline{v}_{i+k}, \cdots, \overline{v}_i, \cdots,\overline{v}_{i+k+p-1}\},
\]
based on  a data-dependent shift $k=k(i)$. This enables the essential non-oscillatory property (\ref{eq:acc}),  while making the ENO procedure \emph{essentially nonlinear}.\newline

\item[Conservation]
The second property sought in the ENO reconstruction is that the piecewise-polynomial ENO approximation be \emph{conservative}, in sense of conserving the original cell averages,
\begin{equation}
\label{eq:interp}
\frac{1}{\hIi}\int_{{I}_i} \mathcal{R}\mathcal{A}v(x) dx = \overline{v}_i.
\end{equation}
The conservative property enables us to recast the ENO procedure in an equivalent formulation of nonlinear \emph{interpolation}. To this end, let $\displaystyle V(x) := \int_{-\infty}^x v(s) ds$ denote the primitive of $v(x)$. The given cell averages $\{\overline{v}_i\}$ now give rise to a set of point-values $\{V_{j+\hf}\}_{j\in\Z}$, 
\begin{equation}\label{eq:primitive}
\begin{aligned}
V_{\jphf} := \int_{-\infty}^{x_{\jphf}} v(s) ds &= \sum_{k=-\infty}^{j} \int_{x_{k-\hf}}^{x_{x+\hf}} v(s) ds = \sum_{k=-\infty}^j |\cell_k| \overline{v}_k.
\end{aligned} 
\end{equation}
A second-order approximation of these point-values is given by the piecewise linear interpolant
${\mathcal L}V(x):=\sum_k \frac{1}{|\cell_k|}\left( V_{k-\hf}(x_{k+\hf}-x) + V_{k+\hf}(x-x_{k-\hf})\right) \Ik(x)$.
The  ENO approximation, $\sum_k \V_k(x) \Ik(x)$, is  a higher-order accurate piecewise-polynomial interpolant,
\begin{equation}\label{eq:enointerpolant}
\text{\bf ENO:} \qquad {\mathcal L}V(x) \quad \mapsto \quad  {\mathcal R}{\mathcal L}V(x):=\sum_k \V_k(x) \Ik(x).
\end{equation}
It interpolates the given data at the nodes, $\V_i(x_{i\pm \hf})=V_{i\pm \hf}$, and it recovers $V(x)$ to high-order accuracy at the interior of the cells, ${\mathcal R}\mathcal{L}V(x)=V(x)+{\mathcal O}(h^{p+1})$. Now, let $\V_i(x)$ be the unique $p$-th order polynomial interpolating the $p+1$ pointvalues $V_{i+r}, \dots, V_{i+r+p}$ for some shift $r$ which is yet to be determined. Then, it is a simple consequence of \eqref{eq:primitive} that $\v$ satisfies \eqref{eq:acc} and \eqref{eq:interp} and that
$$
\V_i(x)=V_{i-\hf}+ \int_{x_{i-\hf}}^x \v_i(s)ds.
$$
In this manner, ENO reconstruction of cell averages is equivalent to ENO interpolation of the pointvalues of its primitive. We shall travel back and forth between these two ENO formulations.
\end{description}

\subsection{ENO reconstruction}
 When the underlying data is sufficiently smooth, the accuracy requirements can be met by interpolating the primitive $V$ on \emph{any} set of $p+1$ point-values 
$$
 \{V_{i+r}, \dots, V_{i-\hf}, V_{i+\hf}, \dots, V_{i+r+p}\}.
$$
Here, $r$ is the (left) \emph{offset} of the interpolation stencil, which is indexed at \emph{half-integers}, to match the cell interfaces. To satisfy the conservation property (\ref{eq:interp}), the stencil of interpolation must include $V_{i-\hf}$ and $V_{i+\hf}$. There are $p$ such  stencils, ranging from the leftmost stencil corresponding to an offset $r=-p+\hf$ to the rightmost stencil corresponding to an offset of $r=-\hf$. 
Since we are interested in approximation of piecewise smooth functions,  we need to choose a carefully shifted stencil, in order to avoid spurious oscillations. The main idea behind the ENO procedure is the use of a   stencil with a data-dependent  offset, $r=r(i)$, which is \emph{adapted to the smoothness of the data}. The choice of ENO stencil is accomplished in  an iterative manner, based on divided differences of the data.
\begin{algorithm}[ENO reconstruction algorithm: selection of ENO stencil]\label{alg:eno}
\mbox{ }
\newline
Let point values of the primitive $V_{i-p+\hf}, \dots, V_{i+p-\hf}$ be given, e.g., \eqref{eq:primitive}.
\begin{itemize}
\item Set $r_1 = -\hf$.
\item For each $j=1, \dots, p-1$, do: 
\[
\begin{cases}
 {\rm if} \quad \left|V\left[x_{i+r_j-1}, \dots, x_{i+r_j+j}\right]\right| < \left|V\left[x_{i+r_j}, \dots, x_{i+r_j+j+1}\right]\right| & \mapsto \quad {\rm set} \quad r_{j+1} = r_j-1,\\ 
 {\rm otherwise} &  \mapsto \quad {\rm set} \quad r_{j+1}=r_j.
\end{cases}
\]
\item Set $\V_i(x)$ as the interpolant of $V$ over the stencil $\{V_{i+k}\}_{k=r_p}^{r_p+p}$.
\item Compute $\v_i(x):=\V_i'(x)$.
\end{itemize}
\end{algorithm}
The divided differences $V[x_k,\dots,x_{k+j}]$ are a good measure of the $j$th order of smoothness of $V(x)$.  Thus, the ENO procedure is based on data-dependent stencils which are chosen in the \emph{direction of smoothness}, in the sense of preferring the smallest divided differences.

The ENO reconstruction procedure was introduced in 1987  by Harten et.~ al.~ \cite{HEOC87} in the context  of accurate simulations for piecewise smooth solutions of nonlinear conservation laws. Since then, the ENO procedure and its extensions, \cite{SO89,Har89,Shu90,Har91,Har93,Har94}, have been used with a considerable success in Computational Fluid Dynamics; we refer to the review article of Shu \cite{Shu97} and the references therein. Moreover, ENO and its various extensions, in particular, with subcell resolution scheme (ENO-SR), \cite{Har89}, have been applied to problems in data compression and image processing in \cite{Har98,AACD02,CZ02,Mat02,BCDS03,CDM03,ACDN05,ACDNM08} and references therein.

There are only a few rigorous results about the global accuracy of the ENO procedure. In \cite{ACDN05}, the authors proved the second-order accuracy of ENO-SR reconstruction of piecewise-smooth $C^2$ data. Multi-dimensional global accuracy results for the so-called ENO-EA method were obtained in \cite{ACDNM08}. Despite the extensive literature on the construction and implementation of ENO method and its variants for the last 25 years, we are not aware of any global, mesh independent, \emph{stability results}. This brings us to the main result of this paper, stating the  stability of the ENO reconstruction procedure in terms of the following sign property.


\begin{theorem}[The sign property]\label{thm:enosign}		
Fix an integer $p>1$. Given the cell averages $\{\overline{v}_i\}$, let ${\mathcal R}{\mathcal A}v(x)$ be the $p$-th order ENO reconstruction of these averages, as outlined in Algorithm \ref{alg:eno},
\[
{\mathcal R}{\mathcal A}v(x) = \sum_k \v_k(x)\Ik(x), \qquad {\rm deg}\,\v_k(x)\leq p-1.
\]
Let $\jumplus{i+\hf}:={\mathcal R}{\mathcal A}v(x_{\iphf}+)$ and $\juminus{i+\hf}:= {\mathcal R}{\mathcal A}v(x_{\iphf}-)$ denote left and right reconstructed point-values at the cell interface $x_{i+\hf}$. Then the following sign property holds at all interfaces:
\begin{equation}\label{eq:signpres}	
\begin{cases}
\ {\rm if} \quad  \overline{v}_{i+1} - \overline{v}_i \geq 0 & \  {\rm then} \quad  \   \jumplus{i+\hf} - \juminus{i+\hf}\geq 0;\\
\ {\rm if} \quad  \overline{v}_{i+1} - \overline{v}_i \leq 0 & \  {\rm then} \quad \    \jumplus{i+\hf} - \juminus{i+\hf}\leq 0.
\end{cases}
\end{equation}
In particular, if $\overline{v}_{i+1} = \overline{v}_i$ then the ENO reconstruction is continuous across the interface, $\jumplus{i+\hf} = \juminus{i+\hf}$. Moreover, there is a constant $\sf{C}_p$, depending only on $p$ and on the mesh-ratio ${|\cell_{j+1}|}/{|\cell_j|}$ of neighboring grid cells, such that
\begin{equation}
\label{eq:jmpub}
0\leq \frac{\jumplus{i+\hf} - \juminus{i+\hf}}{\overline{v}_{i+1} - \overline{v}_i} \leq {\sf C}_p.
\end{equation}
\end{theorem}

The sign property tells us that at each cell interface, the jump of the reconstructed ENO pointvalues cannot have an opposite sign  to the jump in the underlying cell averages. The sign property is illustrated in Figure \ref{fig:eno}, which shows a third-, fourth- and fifth-order ENO reconstruction of randomly chosen cell averages. Even though the reconstructed polynomial may have large variations within each cell, its jumps at cell interfaces always have the same sign as the jumps of the cell averages. Moreover, the relative size of these jumps is uniformly bounded. We remark that the  inequality on the left-hand side of \eqref{eq:jmpub} is a direct consequence of the sign property \eqref{eq:signpres}.

\begin{figure}[h]
\centering
\includegraphics[width=0.32\linewidth]{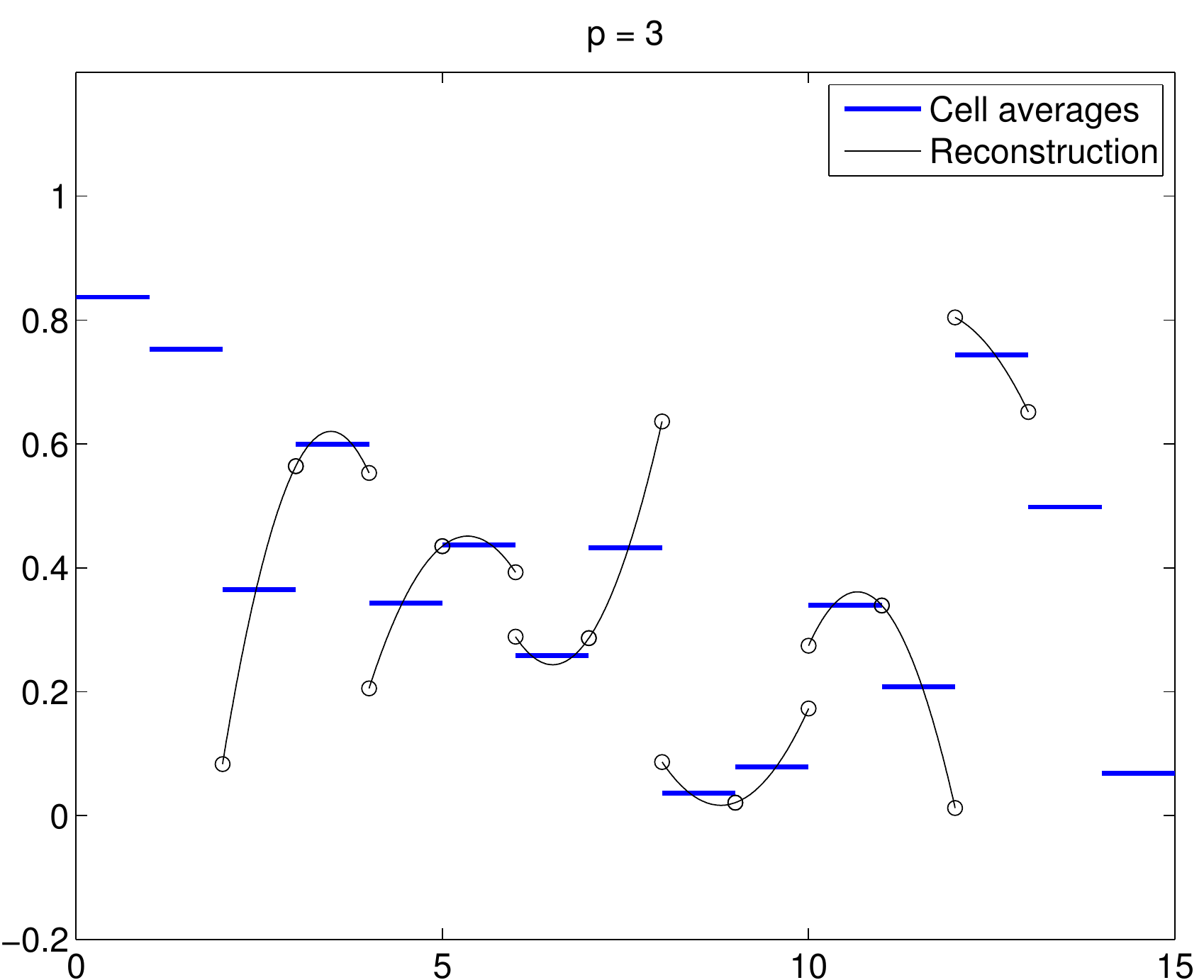}
\includegraphics[width=0.32\linewidth]{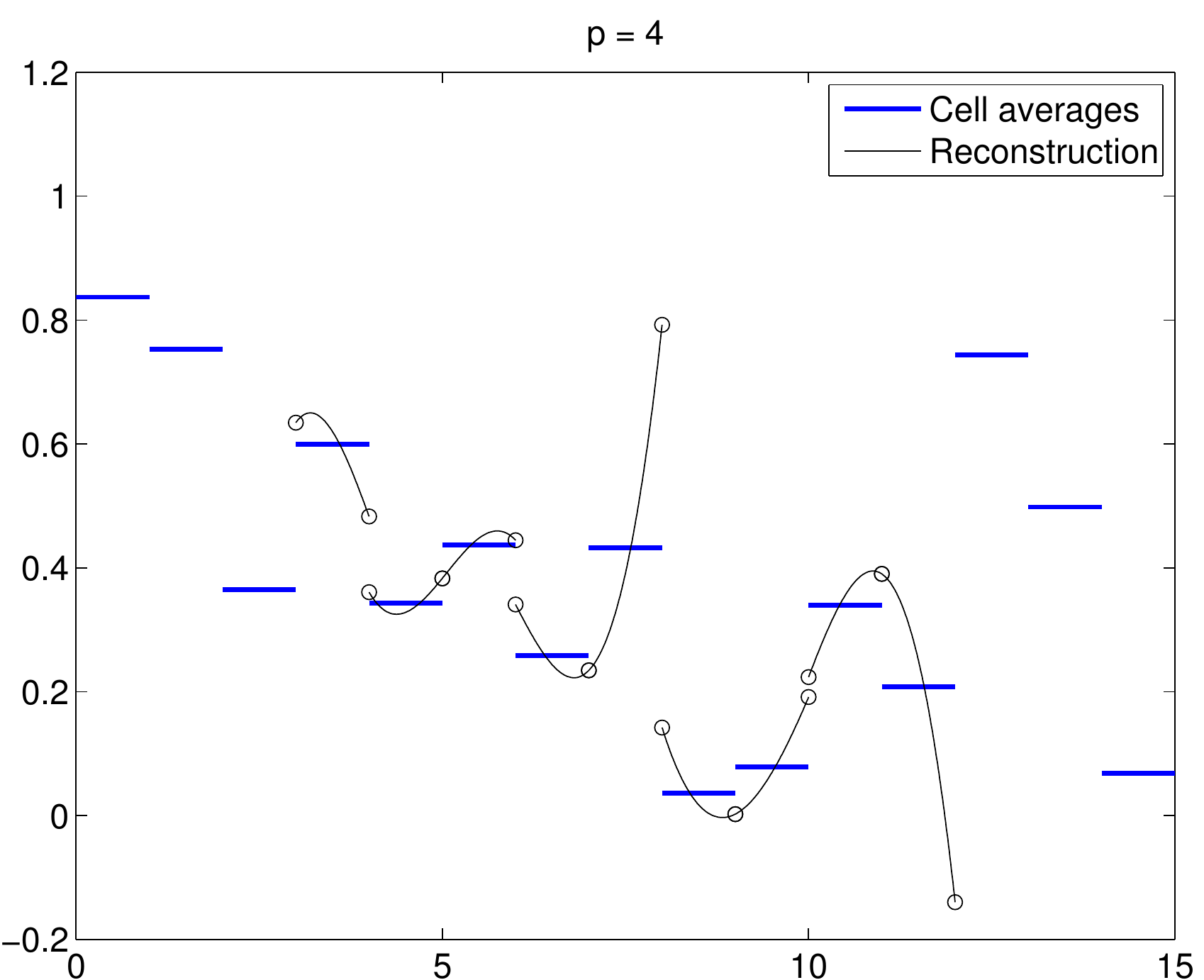}
\includegraphics[width=0.32\linewidth]{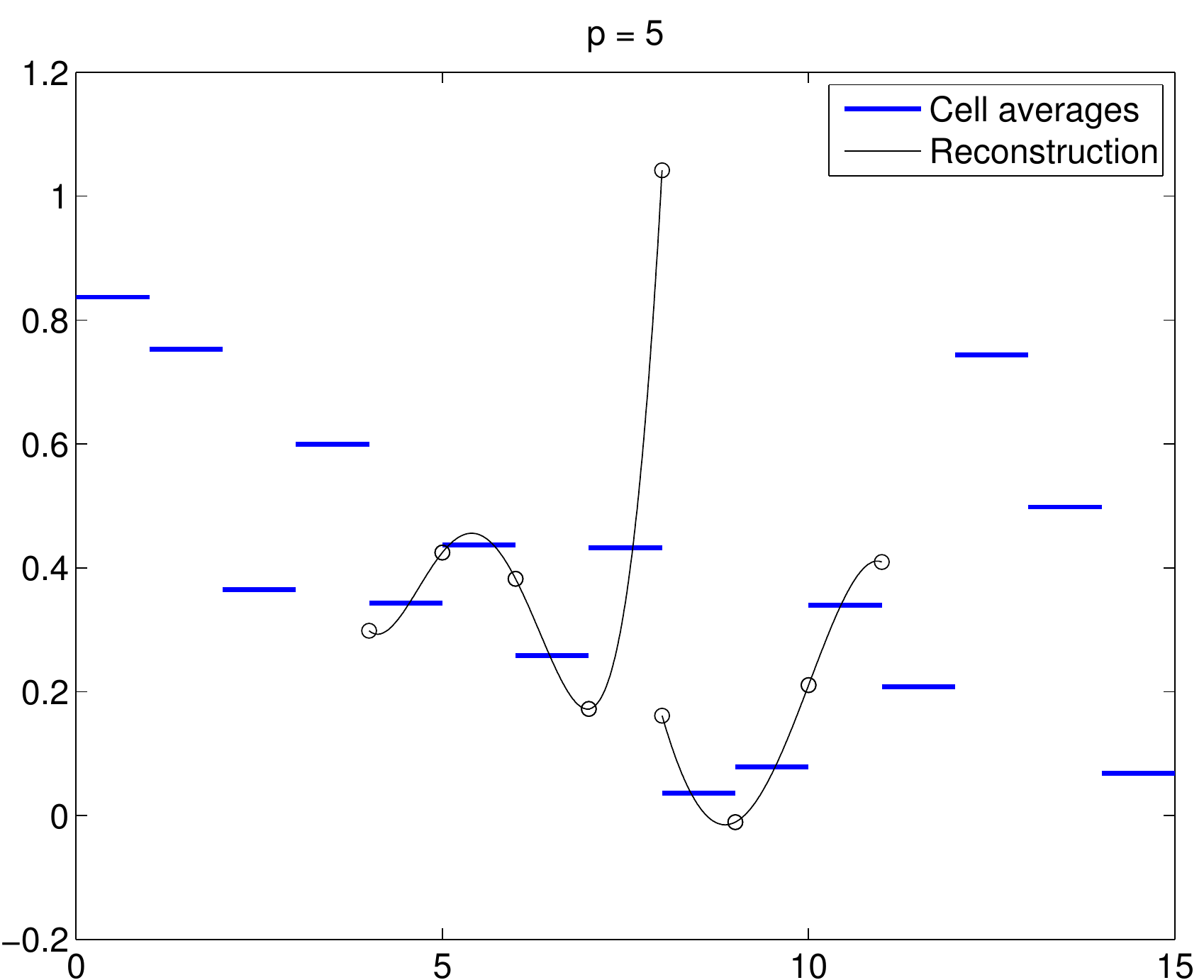}
\caption{ENO reconstruction of randomly chosen cell averages.}
\label{fig:eno}
\end{figure}

\begin{remark}
We emphasize that the main Theorem \ref{thm:enosign} is valid for any order of ENO reconstruction and for any mesh size. It is valid for non-uniform meshes and makes no assumptions on the function $v$, other than that the cell averages $\overline{v}_i$ must be well-defined, which is guaranteed if e.g.~ $v\in L^1_{\rm loc}(\R)$. This is a remarkable rigidity property of the piecewise-polynomial interpolation.
\end{remark}

\begin{remark}
The stability asserted in Theorem \ref{thm:enosign} is realized in terms of the reconstructed point-values at cell interfaces $v^\pm_{i+\hf}$. These are  precisely the input for the construction of high-order accurate finite volume schemes for  nonlinear conservation laws (see Shu \cite{Shu97}), and the relation between these values and the cell averages will be the main point of study in this paper. This approach was taken in \cite{FMT11}, where we use the sign property  to construct arbitrarily high-order accurate \emph{entropy stable} ENO schemes for systems of conservation laws.
\end{remark}

\begin{remark}
The proof of both the sign property and the related upper-bound \eqref{eq:jmpub} depends on the judicious choice of ENO stencils in Algorithm \ref{alg:eno}, and it may fail for other choices of ENO-based algorithms. In particular, the popular  WENO methods, which are based on upwind or central  \emph{weighted} ENO stencils, \cite{JS96,LPR99,QS02},  fail to satisfy the sign property, as can be easily confirmed numerically.
\end{remark}


\subsection{ENO interpolation} 
The ENO algorithm can be formulated as a nonlinear \emph{interpolation} procedure. The starting point is a given collection of point-values, $\{v(x_i)\}_{i\in{\mathbb Z}}$. The purpose of the ENO procedure  in this context (abbreviated by ${\mathcal I}$) is to recover a highly accurate approximation of $v(x)$ from its point-values $v_i:=v(x_i)$,
\[
\text{\bf ENO:} \qquad \{v(x_i)\}_i \quad \ \mapsto \quad \ {\mathcal I}v(x):=\sum_k\v_k(x)\Ik(x), \qquad x_{k\pm\hf}:=\frac{x_k+x_{k\pm1}}{2}.
\]
Here, $\v_k(x)$ are polynomials of degree $p-1$ which interpolate the given data,
\[
{\mathcal I}v(x_i)=\v_i(x_i)=v_i.
\]
Moreover, the ENO interpolant ${\mathcal I}v(x)$ is essentially non-oscillatory in the sense of recovering $v(x)$ to order ${\mathcal O}(h^p)$. In particular, since $v(x)$ may experience jump discontinuities, we wish to recover the point-values, $\juminus{i+\hf}:=\v_i(x_{i+\hf})$ and $\jumplus{i+\hf}:=\v_{i+1}(x_{i+\hf})$, with high-order accuracy,
\[
\left|\juminus{i+\hf}-v(x_{i+\hf}-)\right| +\left|\jumplus{i+\hf}-v(x_{i+\hf}+)\right| = \mathcal{O}(h^p).
\]
This version of the ENO procedure was used for finite difference approximation of nonlinear conservation laws in \cite{SO89}.
The ENO  interpolant ${\mathcal I}v(x)$ is based on the usual divided differences $\{v[x_i,\dots,x_{i+j}]\}_i$, starting with the grid-values $v[x_i]=v_i$ and  defined recursively for $j >0$.

Let $\{\rr_j\}_{j=1}^p$  be the  offsets of the ENO stencil associated with grid point  $x_i$. In this case of ENO interpolation,  these offsets are non-positive \emph{integers}, corresponding to the integer indices of the prescribed gridpoints $x_j$. These offsets are selected according to the following ENO selection procedure.

\begin{algorithm}[ENO interpolation: selection of ENO stencil]\label{alg:eno_interpolant}
\mbox{ }
\newline
Let point values $v_{i-p+1}, \dots, v_{i+p-1}$ be given.
\begin{itemize}
\item Set $\rr_1 = 0$.
\item For each $j=1, \dots, p-1$, do: 
$$
\begin{cases}
\text{if} \quad |v[x_{i+\rr_{j}-1},\dots,x_{i+\rr_j+j-1}]| < |v[x_{i+\rr_j},\dots,x_{i+\rr_j+j}]| & \mapsto \quad {\rm set} \ \ \ \rr_{j+1}=\rr_j-1, \\
\text{otherwise}& \mapsto \quad {\rm set} \ \ \ \rr_{j+1}=\rr_j.
\end{cases}
$$

\item Set $\v_i(x)$ as the interpolant of $v$ over the stencil $\{v_{i+k}\}_{k=\rr_p}^{\rr_p+p-1}$:
\[
\v_i(x) = \sum_{j=0}^{p-1} v[x_{i+\rr_j},\dots,x_{i+\rr_j+j}] \prod_{m=0}^{j-1} \left(x-x_{i+\rr_{j}+m}\right).
\]
\end{itemize}
\end{algorithm}

 In the following theorem we state the main stability result for this version of the ENO interpolation procedure, analogous to the sign property of the ENO reconstruction procedure from cell averages. 

\begin{theorem}[The sign property revisited -- ENO interpolation]\label{thm:enosignp}		
Fix an integer $p>1$. Given the point-values $\{{v}_i\}$, let ${\mathcal I}v(x)$ be the $p$-th order ENO interpolant of these point-values, outlined in Algorithm \ref{alg:eno_interpolant},
\[
{\mathcal I}v(x) = \sum_k \v_k(x)\Ik(x), \qquad {\rm deg}\,\v_k(x)\leq p-1.
\]
Let $\juminus{i+\hf}:={\mathcal I}v(x_{\iphf}-)$ and $\jumplus{i+\hf}:= {\mathcal I}v(x_\iphf+)$ denote left and right reconstructed point-values at the cell interfaces $x_{i+\hf}$. Then the following sign property holds at all interfaces:
\begin{equation}\label{eq:signpresPW}	
\begin{cases}
\ {\rm if} \quad  {v}_{i+1} - {v}_i \geq 0 & \  {\rm then} \quad   \jumplus{i+\hf} - \juminus{i+\hf}\geq 0;\\
\ {\rm if} \quad  {v}_{i+1} - {v}_i \leq 0 & \  {\rm then} \quad    \jumplus{i+\hf} - \juminus{i+\hf}\leq 0.
\end{cases}
\end{equation}
In particular, if the point-values ${v}_{i+1} = {v}_i$ then the ENO interpolation is continuous across their mid-points, $\jumplus{i+\hf} = \juminus{i+\hf}$. 
Moreover, there is a constant ${\sf c}_p$, depending only on $p$ and on the mesh-ratio ${|x_{j+1}-x_j|}/{|x_j-x_{j-1}|}$ of neighboring grid cells, such that
\begin{equation}
\label{eq:jmpubPW} 	
0\leq \frac{\jumplus{i+\hf} - \juminus{i+\hf}}{{v}_{i+1} - {v}_i} \leq {\sf c}_p.
\end{equation}
\end{theorem}

The rest of this paper is devoted to proving the stability of the ENO procedure. We begin with the ENO reconstruction procedure in Theorem \ref{thm:enosign}. In Section \ref{sec:sign} we prove the sign property \eqref{eq:signpres} and in Section \ref{sec:jumpBound} we prove the upper bound \eqref{eq:jmpub}. Similar stability results holds for the ENO interpolation procedure, dealing with  point-values instead of cell averages. In Section \ref{sec:fdm} we prove the sign property for the ENO interpolation stated in the main theorem \ref{thm:enosignp}.\newline
These results were announced earlier in \cite{FMT10}.

\section{The sign property for ENO reconstruction}\label{sec:sign}
The aim of this section is to prove the sign property \eqref{eq:signpres}. 
To this end we derive a novel expression of  the interface jump, $\jumplus{i+\hf}-\juminus{i+\hf}$, as a  sum of terms which involve  $(p+1)$-th divided differences of $V$, and we show that each summand in this expression has the same sign as $\overline{v}_{i+1}-\overline{v}_i$. 

We recall that at each cell $\cell_i$, the ENO reconstruction is based on a particular stencil of $p+1$ consecutive gridpoints, $\{x_{i+r}, \dots, x_{i+r+p}\}$, where $r=r(i)$ is the  offset of such stencil. 

\bigskip\noindent
{\bf Notation}. We will reserve the indices $r$ and $s$ to denote  offsets of ENO reconstruction stencils. We recall that these offsets measure the shifts to the left of each stencil, and are indexed at negative \emph{half-integers}, $-p+\hf \leq r,s \leq -\hf$,  to match the indexing of cell interfaces at half-integers. 
\bigskip

Since this choice of ENO offset depend on the data through the iterative Algorithm \ref{alg:eno}, we need to trace the hierarchy of ENO stencils which ends with the final offset $r=r(i)$. To simplify notations, we focus our attention on a typical cell $\cell_0$, with an initial stencil which consists of the edges at $x_{-\hf}$ and $x_\hf$. The stencil is identified by its leftmost index, $r_1=-\hf$. Next, the stencil is extended, either to the left, $\{x_{-\thf},x_{-\hf},x_\hf\}$ where $r_2=-\thf$, or to the right, $\{x_{-\hf},x_\hf,x_{\thf}\}$ with $r_2=-\hf$. In the next stage, there are three possible stencils, which are identified by the leftmost offset:  $r_3=-\fhf$ corresponding to $\{x_{-\fhf},\dots,x_\hf\}$,  $r_3=-\thf$ corresponding to $\{x_{-\thf},\dots,x_\thf\}$, or $r_3=-\hf$ corresponding to $\{x_{-\hf},\dots,x_\fhf\}$. Stage $j$ of the ENO Algorithm \ref{alg:eno} involves the stencil of $j+1$ consecutive points, $\{x_{r_j},\dots,x_{r_j+j}\}$. The series of offsets of this hierarchy of stencils,
$r_1,r_2,\dots,r_p$, forms the \emph{signature} of the ENO algorithm. Note that by our construction,
\[
r_1=-\hf\geq r_2\geq r_3\geq \dots \geq r_p \geq -p+\hf,
\] 
and whenever needed, we set  $r_{-1} = r_0 = -\hf$. The stability of ENO will be proved by carefully studying such data-dependent signatures. 
 
The Newton representation of the $p$-th degree interpolant $F_0(x)$, based on point-values \\
$V(x_{r_p}),V(x_{r_p+1}),\dots,V(x_{r_p+p})$, is given by 
\begin{equation}\label{eq:Newton}
\V_0(x) = \sum_{j=0}^p V\left[x_{r_j}, \dots, x_{r_j + j}\right] \prod_{m=0}^{j-1}\left(x-x_{r_{j} + m}\right),
\end{equation}
where $V[x_k,\dots,x_{k+j}]$ are the $j$-th  divided difference of $V$ at the specified gridpoints. Observe that in (\ref{eq:Newton}), we took the liberty of  summing the contributions of stencils in their ``order of  appearance' rather than the usual sum of stencils from left to right.

\bigskip\noindent
{\bf Notation}. For notational convenience, we denote the $j$-th divided difference of the primitive $V$ as
\[
\Del{r}{j} := V[x_r, \dots, x_{r+j}], \quad \Del{r}{j}= \frac{\Delii{r+1}{r+j}-\Del{r}{j-1}}{x_{r+j}-x_r}, \qquad r=\dots,-\thf,-\hf,\hf.
\]
Thus, for example, by \eqref{eq:primitive} we have $\Delii{-\hf}{\thf} =V[x_{-\hf},x_\hf,x_\thf]= (x_{\thf}-x_{-\hf})(\overline{v}_1-\overline{v}_0)$. 

\medskip
If $\Delii{-\hf}{\thf} = 0$, or in other words, if $\avg{v}_0 = \avg{v}_1$, then it is easy to see that the ENO procedure will end up with identical stencils for $\cell_0$ and $\cell_1$, which in turn yields $\vminus = \vplus$. We may therefore assume that $\Delii{-\hf}{\thf} \neq 0$, and the sign property will be proved by showing that 
\begin{equation}\label{eq:sprop}
\text{\bf Sign property:} \qquad 
\begin{cases}
\ {\rm if} \quad  \Delii{-\hf}{\thf}> 0 \ \ & {\rm then} \quad  \vplus-\vminus\geq 0;\\
\ {\rm if} \quad \Delii{-\hf}{\thf}< 0 \ \ & {\rm then} \quad  \vplus-\vminus\leq 0.
\end{cases}
\end{equation}

To verify (\ref{eq:sprop}), we examine the ENO reconstruction at cell $\cell_0$, given by  $\v_0(x) := \V_0'(x)$. Differentiation of (\ref{eq:Newton}) yields 
\[
\v_0(x) = \sum_{j=1}^p \Del{r_j}{j} \sum_{l=0}^{j-1}\prod_{\substack{m=0\\m\neq l}}^{j-1}\left(x-x_{r_{j} + m}\right).
\]
The value of $\v_0$ at the cell interface $x_\hf$ is then
\begin{equation}\label{eq:stam}
\vminus = \v_0(x_\hf) = \sum_{j=1}^p \Del{r_j}{j} \sum_{l=0}^{j-1}\prod_{\substack{m=0\\m\neq l}}^{j-1}\left(x_\hf - x_{r_{j} + m}\right) 
= \sum_{j=1}^p \Del{r_j}{j}\!\!\!\!\!\!\!\!\!\!\!\!\prod_{\substack{m=0\\m\neq -r_{j}+\hf}}^{j-1}\!\!\!\!\!\!\! \left(x_\hf-x_{r_{j} + m}\right) 
\end{equation}
The last equality follows from the fact that all  but the one term corresponding to $l=-r_{j}+\hf$ drop out. 

\bigskip\noindent
{\bf Notation}. To simplify notations, we use $\prodii{}{}$ to denote a product which skips any of its zero factors,
$ \prodii{j\in J}{}\alpha_j:= \prod_{j\in J: \alpha_j\neq 0}\alpha_j$.\newline
Thus, for example, a simple shift of indices in (\ref{eq:stam}) yields
$\displaystyle 
\vminus =  \sum_{j=1}^p \Del{r_j}{j} \prodii{m=-1}{j-2}\left(x_\hf-x_{r_{j} + m+1}\right)$.
In an similar fashion, we handle the ENO reconstruction at cell $\cell_1$. Let $s_1, \dots, s_p$ be the signature of that cell. Note that $r_j \leq s_j+1$, since the ENO reconstruction at stage $j$ in cell $\cell_1$ cannot select a stencil further to the left than the one used in cell $\cell_0$. If $r_j = s_j+1$, then the two interpolation stencils are the same, and so $\vplus-\vminus=0$. Hence, we only need to consider the case $r_j\leq s_j$. The reconstructed value of $\v_1(x)=\V_1'(x)$ at $x=x_\hf$ is  given by 
\[
\vplus=f_1(x_{\hf})=\sum_{j=1}^p \Del{1+s_j}{j} \prodii{m=0}{j-1}\left(x_\hf-x_{s_{j} + m+1}\right)  
\]
The jump in the  values reconstructed at $x=x_\hf$ is then given by 
\begin{equation}\label{eq:jump}
\vplus-\vminus = \sum_{j=1}^p\left(\Del{1+s_j}{j} \prodii{m=0}{j-1}\left(x_{\hf}-x_{s_{j} + m+1}\right) - \Del{r_j}{j} \prodii{m=-1}{j-2}\left(x_\hf-x_{r_{j} + m+1}\right)\right). 
\end{equation}

The following lemma provides a much needed simplification for the rather intimidating expression (\ref{eq:jump}), in terms of a key identity, which is interesting in its own right.

\begin{lemma}\label{thm:sumexpr}		
The jump of the reconstructed point-values in \eqref{eq:jump} is given by  
\begin{equation}\label{eq:sumexpr}		
\begin{split}
\vplus-\vminus = \sum_{r=r_p}^{s_p}\Del{r}{p+1}(x_{r+p+1} - x_{r})\prodii{m=0}{p-1}\left(x_\hf-x_{r + m+1}\right).  
\end{split}
\end{equation}
\end{lemma}

We postpone the proof of Lemma \ref{thm:sumexpr} to the end of this section, and we turn to use it in order to conclude the proof of the sign property. 
To this end,  we show that each non-zero summand in \eqref{eq:sumexpr} has the same sign as $\avg{v}_1-\avg{v}_0$. Since
\begin{subequations}\label{eqs:sign}
\begin{equation}\label{eq:signa}
\sgn\left((x_{r+p+1} - x_{r})\prodii{m=0}{p-1}\left(x_\hf-x_{r + m+1}\right)\right) = (-1)^{r+p-\hf}, \quad r=-p+\hf,\dots,-\hf,   
\end{equation}
then in view of (\ref{eq:sprop}), it remains to prove the following.
\begin{lemma}\label{thm:ddsign}		
Let $\{r_j\}_{j=1}^{p}$ and $\{s_j\}_{j=1}^{p}$ be the signatures of the ENO stencils associated with cells $\cell_0$ and, respectively, $\cell_1$. Then the following holds:
\begin{eqnarray}
\ {\rm if } \quad \Delii{-\hf}{\thf} > 0 & {\rm then} \quad (-1)^{r+p-\hf} \Del{r}{p+1}\geq 0,   \qquad r=r_p, \dots,s_p,  \label{eq:signb}\\
\ {\rm if } \quad \Delii{-\hf}{\thf} < 0 & {\rm then} \quad (-1)^{r+p-\hf} \Del{r}{p+1}\leq 0,
 \qquad  r=r_p, \dots,s_p.  \label{eq:signc}
\end{eqnarray}
\end{lemma}
\end{subequations}
Since $r$ runs over half-integers, (\ref{eqs:sign}) imply that each non-zero term in the sum \eqref{eq:sumexpr}  has the same sign as $\Delii{-\hf}{\thf}$, and Theorem \ref{thm:enosign} follows from the sign property, (\ref{eq:sprop}). 

\begin{proof}
We consider the case (\ref{eq:signb}) where $\Delii{-\hf}{\thf} > 0$;  the case (\ref{eq:signc}) can be argued similarly. The result clearly holds for $p=1$, where $r_p=s_p=-\hf$. Assuming that it holds for some $p\geq 1$, namely,  that  $(-1)^{r+p-\hf}\Del{r}{p+1}\geq 0$ for  $r=r_p,\dots,s_p$,  we will verify that it holds for $p+1$. Indeed, 
\begin{eqnarray}\label{eq:already}
(-1)^{r+p+\hf}\Del{r}{p+2} &\equiv& (-1)^{r+p+\hf}\frac{\Delii{r+1}{r+1+p+1} - \Del{r}{p+1}}{x_{r+p+2} - x_{r}}\\
 & = & \frac{(-1)^{r+p+\hf}\Delii{r+1}{r+p+2} + (-1)^{r+p-\hf} \Del{r}{p+1}}{x_{r+p+2} - x_{r}} \geq 0 \nonumber 
\end{eqnarray}
for $r=r_p,\dots,s_p-1$, by the induction hypothesis. Thus, it remains to examine $\Del{r}{p+2}$ when $r=r_{p+1}< r_p$ and $r=s_{p+1}\geq s_p$.  
\begin{itemize}
 \item[(a)] The case $r_{p+1} = r_p$ is already included in (\ref{eq:already}), so the only other possibility is when $r_{p+1}=r_p-1$. According to the ENO selection principle, this choice of extending the stencil to the left occurs when $|\Delii{r_p-1}{r_p+p}| < |\Del{r_p}{p+1}|$. Consequently, since $(-1)^{r_{p+1}+p+\hf}=(-1)^{r_{p}+p-\hf}$, and by assumption
$(-1)^{r_{p}+p-\hf}\Del{r_p}{p+1}=|\Del{r_p}{p+1}|$, we have 
\[
(-1)^{r_{p+1}+p+\hf} \Del{r_{p+1}}{p+2} \geq  \frac{(-1)^{r_{p}+p-\hf}\Del{r_p}{p+1} - |\Delii{r_p-1}{r_p+p}|}{{x_{r_p+p+1} - x_{r_p-1}}} > 0,
\]
and (\ref{eq:signb}) follows.
\item[(b)] The case $s_{p+1}=s_p-1$ is already covered in (\ref{eq:already}). The only other possibility is therefore  $s_{p+1} = s_p$. By the ENO selection procedure, this extension to the right occurs when
$|\Delii{s_p+1}{s_p+p+2}| \leq |\Del{s_p}{p+1}|$. But $(-1)^{s_{p+1}+p+\hf}=-(-1)^{s_p+p-\hf}$ and by assumption, $(-1)^{s_p+p-\hf} \Del{s_p}{p+1}=|\Del{s_p}{p+1}|$, hence
\[
(-1)^{s_{p+1}+p+\hf} \Del{s_{p+1}}{p+2} \geq \frac{-|\Delii{s_p+1}{s_p+p+2}| + (-1)^{s_p+p-\hf} \Del{s_p}{p+1}}{{x_{s_p+p+2} - x_{s_p}}} \geq0,
\]
 and (\ref{eq:signb}) follows. 
\end{itemize}
\end{proof}

We close this section with the promised \emph{proof of Lemma \ref{thm:sumexpr}}.
\begin{proof}  
We proceed in two steps.
In the first step, we  consider the special case when the two stencils that are used by the ENO reconstruction in cells $\cell_0$ and $\cell_1$ are only one grid cell apart.
Such stencils must have the same offset, say $r_p = s_p = r$ and in this case, Lemma \ref{thm:sumexpr} claims that $\vplus-\vminus$ equals
\begin{equation}\label{eq:singleJumpExpr}
f_1(x_\hf)-f_0(x_\hf) = \Del{r}{p+1}\left(x_{r+p+1} - x_{r}\right)\prodii{m=0}{p-1}\left(x_\hf-x_{r + m+1}\right).  
\end{equation}
Indeed, the interpolant of $V(x_{r+1}),\dots,V(x_{r+p}),V(x_{r})$, assembled in the specified  order from left to right and then adding $V(x_{r})$ at the end, is given by
\[
F_0(x)=\sum_{j=0}^{p-1}\Delii{r+1}{r+1+j}\prod_{m=0}^{j-1}\left(x-x_{r+1+m}\right)
+\Del{r}{p}\prod_{m=0}^{p-1}\left(x-x_{r+1+m}\right).
\]
Similarly, the interpolant of $V(x_{r+1}),\dots,V(x_{r+p+1})$, assembled in the specified order from left to right, is given by
\[
F_1(x)=\sum_{j=0}^{p-1}\Delii{r+1}{r+j+1}\prod_{m=0}^{j-1}\left(x-x_{r+1+m}\right) + \Delii{r+1}{r+p+1}\prod_{m=0}^{p-1}\left(x-x_{r+1+m}\right)
\]
(cf. \eqref{eq:Newton}). Thus, their difference amounts to
\begin{align*}
F_1(x)-F_0(x) &= \left(\Delii{r+1}{r+p+1}- \Del{r}{p}\right)\prod_{m=0}^{p-1}\left(x-x_{r+1+m}\right) \\
&= \Delii{r}{r+p+1}\left(x_{r+p+1}-x_r\right)\prod_{m=0}^{p-1}\left(x-x_{r+1+m}\right),
\end{align*}
which reflects the fact that $F_0$ and $F_1$  coincide at the $p$ points $x_{r+1},\dots,x_{r+p}$. Differentiation yields
\begin{align*}
f_1(x)-f_0(x) &= \Delii{r}{r+p+1}\left(x_{r+p+1}-x_r\right)\sum_{l=0}^{p-1}\prod_{\substack{m=0\\m\neq l}}^{p-1}\left(x-x_{r+1+m}\right).
\end{align*}
At $x=x_\hf$, all product terms on the right vanish except for $l=-r-\hf$, since $x_\hf$ belongs to both stencils. We end up with
\[
f_1(x_\hf)-f_0(x_\hf)= \Delii{r}{r+p+1}\left(x_{r+p+1}-x_r\right)\prodii{m=0}{p-1}\left(x_\hf-x_{r+m+1}\right).
\]
This shows that (\ref{eq:singleJumpExpr}) holds, verifying Lemma \ref{thm:sumexpr} in the case that the stencils associated with $\cell_0$ and $\cell_1$ are separated by just one cell.

In step two, we extend this result for arbitrary stencils, where $\cell_0$ and $\cell_1$ are associated with arbitrary offsets, $r_p\leq s_p$. Denote by $F^{\{j\}}$ the interpolant at points $x_j,\dots,x_{j+p}$, so that $\V_0 = F^{\{r_p\}}$ and $\V_1 = F^{\{s_p+1\}}$. Using the representation from the first step for the difference of one-cell shifted stencils, $\bigl(f^{\{r+1\}} - f^{\{r\}}\bigr)(x_\hf)$, we can write the  jump at the cell interface as a telescoping sum,
\begin{align*}
(\v_1 - \v_0)(x_\hf) &= \left(f^{\{s_p+1\}} - f^{\{r_p\}}\right)(x_\hf) \\
&= \sum_{r=r_p}^{s_p} \Bigl(f^{\{r+1\}} - f^{\{r\}}\Bigr)(x_\hf) \\
&= \sum_{r=r_p}^{s_p}\Del{r}{p+1}(x_{r+p+1} - x_{r})\prodii{m=0}{p-1}\left(x_\hf-x_{r + m+1}\right),  
\end{align*}
and (\ref{eq:sumexpr}) follows.
\end{proof}

\section{The relative jumps in ENO reconstruction are bounded}\label{sec:jumpBound}
In this section, we will prove \eqref{eq:jmpub}, which establishes an upper bound on the size of the jump in reconstructed values in terms of the jump in the underlying cell averages. We need the following lemma.

\begin{lemma}\label{thm:ddbound}
Let $r_p, s_p$ be the (half-integer) offsets of the ENO stencils associated with cell $\cell_0$ and, respectively, $\cell_1$. Then
\begin{subequations}\label{eqs:ubounds}
\begin{equation}\label{eq:ubounda}
\frac{\Del{r}{p+1}}{\Delii{-\hf}{\thf}}(-1)^{r+p-\hf} \leq \Cu{r}{p}, \qquad r = r_p, \dots, s_p,
\end{equation}
where the constants $\Cu{r}{p}$ are defined recursively, starting with  $\Cu{r}{1} = 1$, and
\begin{equation}\label{eq:uboundb}
\Cu{r}{p+1} =\frac{2}{x_{r+p+2}-x_r} \max\left(C_{r, p}, C_{r+1, p}\right) \qquad \forall ~ r.
\end{equation}
\end{subequations}
\end{lemma}
The quantity on the left in \eqref{eq:ubounda} was shown to be bounded from below by zero in Lemma \ref{thm:ddsign}; here we prove an upper bound. The constants $C_{r,p}$ \emph{only} depend on the grid sizes $|I_j|$.
\begin{proof}
The result clearly holds for $p=1$. We prove the general induction step passing from $p\mapsto p+1$. Using the recursion relation  
\[
\Del{r}{p+2} = \frac{\Delii{r+1}{r+p+2} - \Del{r}{p+1}}{\Delta x}, \qquad \Delta x:= x_{r+p+2}-x_r,
\]
we have
\begin{equation}\label{eq:alreadyii}
\begin{split}
0 \leq \frac{\Del{r}{p+2}}{\Delii{-\hf}{\thf}}(-1)^{r+p+\hf} &= 
 \frac{1}{\Delta x}\left(\frac{\Delii{r+1}{(r+1)+p+1}}{\Delii{-\hf}{\thf}}(-1)^{r+p+\hf} + \frac{\Delii{r}{r+p+1}}{\Delii{-\hf}{\thf}}(-1)^{r+p-\hf}\right) \\
&\leq \frac{\Cu{r+1}{p}+\Cu{r}{p}}{\Delta x} \leq \Cu{r}{p+1}, \qquad r = r_p, \dots, s_p-1.
\end{split}
\end{equation}
We turn to the remaining cases. 
\begin{itemize}
\item[(a)] As in Lemma \ref{thm:ddsign}, if $r_{p+1} = r_p$ then the induction step is already covered in (\ref{eq:alreadyii}), so the only remaining case is $r=r_{p+1}=r_p-1$, corresponding to Lemma \ref{thm:ddsign}(a).
In this case, $|\Del{r}{p+1}| \leq |\Delii{r+1}{r+p+2}|$, hence
\begin{align*}
\frac{\Del{r}{p+2}}{\Delii{-\hf}{\thf}}(-1)^{r+p+\hf} & = \frac{1}{\Delta x}\frac{\Delii{r+1}{r+p+2} - \Del{r}{p+1}}{\Delii{-\hf}{\thf}}(-1)^{r+p+\hf} \\
& \leq \frac{2}{\Delta x}\frac{\Delii{r+1}{r+p+2}}{\Delii{-\hf}{\thf}}(-1)^{r+p+\hf} 
\leq \frac{2\Cu{r+1}{p}}{\Delta x} \leq \Cu{r}{p+1}.
\end{align*}

\item[(b)] If $s_{p+1} = s_p-1$ then the induction step is already covered in (\ref{eq:alreadyii}), so the only remaining case is $r=s_{p+1}=s_p$, corresponding to Lemma \ref{thm:ddsign}(b).
In this case, $|\Delii{r+1}{r+p+2}| \leq |\Del{r}{p+1}|$, hence
\begin{align*}
\frac{\Del{r}{p+2}}{\Delii{-\hf}{\thf}}(-1)^{r+p+\hf} &= \frac{1}{\Delta x}\frac{\Delii{r+1}{r+p+2} - \Del{r}{p+1}}{\Delii{-\hf}{\thf}}(-1)^{r+p+\hf} \\
&\leq \frac{2}{\Delta x}\frac{\Del{r}{p+1}}{\Delii{-\hf}{\thf}}(-1)^{r+p-\hf}
\leq \frac{2\Cu{r}{p}}{\Delta x} \leq \Cu{r}{p+1}.
\end{align*}
\end{itemize}
\end{proof}

Using the explicit form \eqref{eq:sumexpr} of the jump $\vplus-\vminus$, we get the following explicit expression of the upper-bound asserted in \eqref{eq:jmpub}.
\begin{theorem}\label{thm:jmppr}
Let $\vplus$ and $\vminus$ be the point-values reconstructed by the ENO algorithm \ref{alg:eno} at the cell interface $x=x_{\hf}+$ and, respectively, $x=x_{\hf}-$.  Then
\begin{align*}
\frac{\vplus-\vminus}{\overline{v}_1-\overline{v}_0} 
\leq {\sf C}_p := \frac{1}{x_\thf-x_{-\hf}}\sum_{r=-p+\hf}^{-\hf}\Cu{r}{p} \left|(x_{r+p+1} - x_r)\prodii{m=0}{p-1}\left(x_\hf-x_{r+m+1}\right)\right|.
\end{align*}
\end{theorem}
\begin{proof}
Let $r_p, s_p$ be the offsets of the ENO stencils associated with cell $\cell_0$ and, respectively, $\cell_1$. By Lemmas \ref{thm:sumexpr} and \ref{thm:ddbound}, we have
\begin{align*}
\frac{\vplus-\vminus}{\overline{v}_1-\overline{v}_0} &= \frac{1}{x_\thf-x_{-\hf}}\sum_{r=r_p}^{s_p}\frac{\Del{r}{p+1}}{\Delii{-\hf}{\thf}}(-1)^{r+p-\hf}\left|(x_{r+p+1} - x_r)\prodii{m=0}{p-1}\left(x_\hf-x_{r + m+1}\right)\right| \\
&\leq \frac{1}{x_\thf-x_{-\hf}}\sum_{r=r_p}^{s_p}\Cu{r}{p}\left|(x_{r+p+1} - x)\prodii{m=0}{p-1}\left(x_\hf-x_{r+m+1}\right)\right| \\
&\leq \frac{1}{x_\thf-x_{-\hf}}\sum_{r=-p+\hf}^{-\hf}\Cu{r}{p}\left|(x_{r+p+1} - x_r)\prodii{m=0}{p-1}\left(x_\hf-x_{r+m+1}\right)\right|.
\end{align*}
\end{proof}

When the mesh is uniform, $|\cell_i| \equiv h$, the expression for the upper bound ${\sf C}$ can be calculated explicitly. The recursion relation \eqref{eq:uboundb} yields
$\displaystyle \Cu{r}{p} = \frac{2^p}{h^{p-1}(p+1)!}$, and the coefficient of the $(p+1)$-th order divided differences in \eqref{eq:sumexpr} is
\[
\left|(x_{r+p+1} - x_r)\prodii{m=0}{p-1}\left(x_\hf-x_{r+m+1}\right)\right| = h^{p}(p+1) (-r-\hf)!(p+r-\hf)!.
\]
Thus, we arrive at the following bound on the jump in reconstructed values.
\begin{corollary}
Let $\vplus$ and $\vminus$ be the pointvalues reconstructed at the cell interface $x=x_{\hf}+$ and, respectively, $x=x_{\hf}-$, by the $p$-th order ENO Algorithm \ref{alg:eno}, based on equi-spaced cell averages $\{\overline{v}_k\}_k$. Then
\begin{equation}\label{eq:jumpbound}
\begin{split}
\frac{\vplus-\vminus}{\overline{v}_1-\overline{v}_0} \leq {\sf C}_p = 2^{p-1}\frac{1}{p!}\sum_{k=0}^{p-1} k!(p-k-1)!.
\end{split}
\end{equation}
\end{corollary}
Table \ref{tab:upperBound} shows the upper bound \eqref{eq:jumpbound} on $(\vplus-\vminus) / (\overline{v}_1-\overline{v}_0)$ for some values of $p$. Returning to Figure \ref{fig:eno}, we see that although the jumps at the cell interfaces can get large, they \emph{cannot} exceed ${\sf C}_p$ times the size of the jump in cell averages, regardless of the values in neighboring cell averages.
\begin{table}[ht]
\begin{center}
\begin{tabular}{|c|c|}
\hline $p$ & Upper bound ${\sf C}_p$ \\
\hline 1 & 1 \\
\hline 2 & 2 \\
\hline 3 & $10/3 = 3.333\dots$ \\
\hline 4 & $16/3 = 5.333\dots$ \\
\hline 5 & $128/15 = 8.533\dots$ \\
\hline 6 & $208/15 = 13.866\dots$ \\
\hline
\end{tabular}
\end{center}
\caption{}
\label{tab:upperBound}
\end{table}

\begin{figure}[h]
\centering
\includegraphics[width=0.4\linewidth]{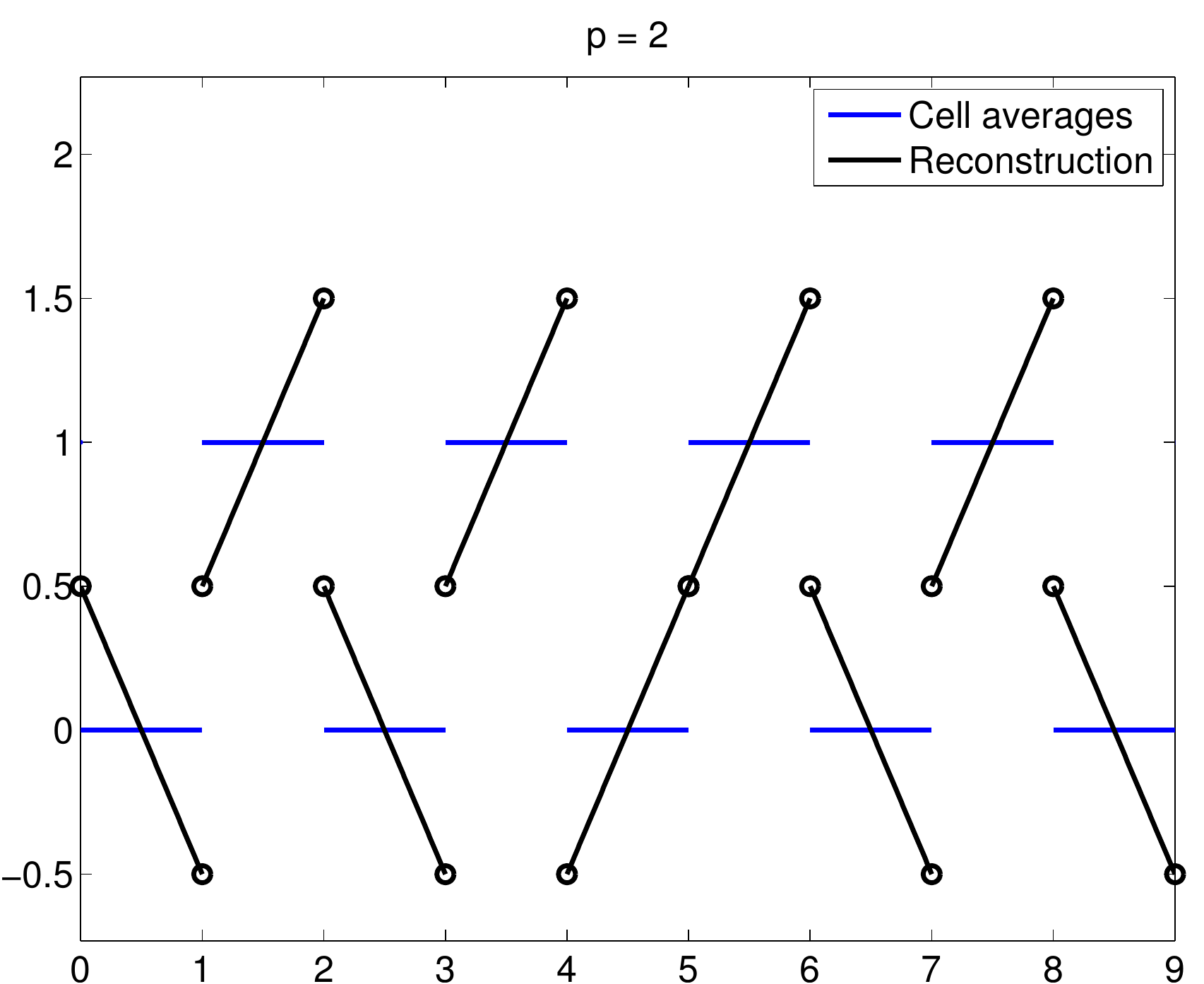}
\includegraphics[width=0.4\linewidth]{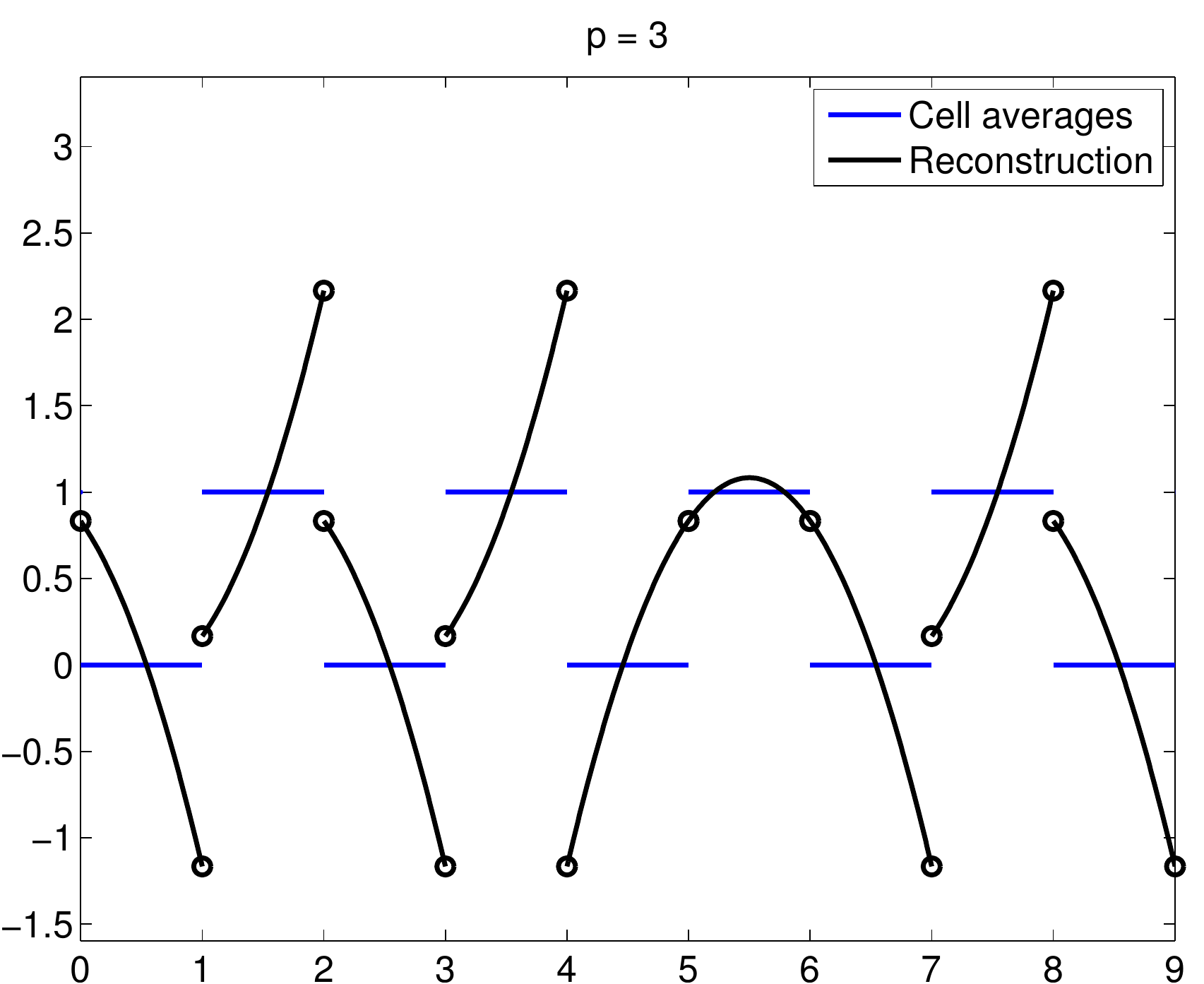}
\includegraphics[width=0.4\linewidth]{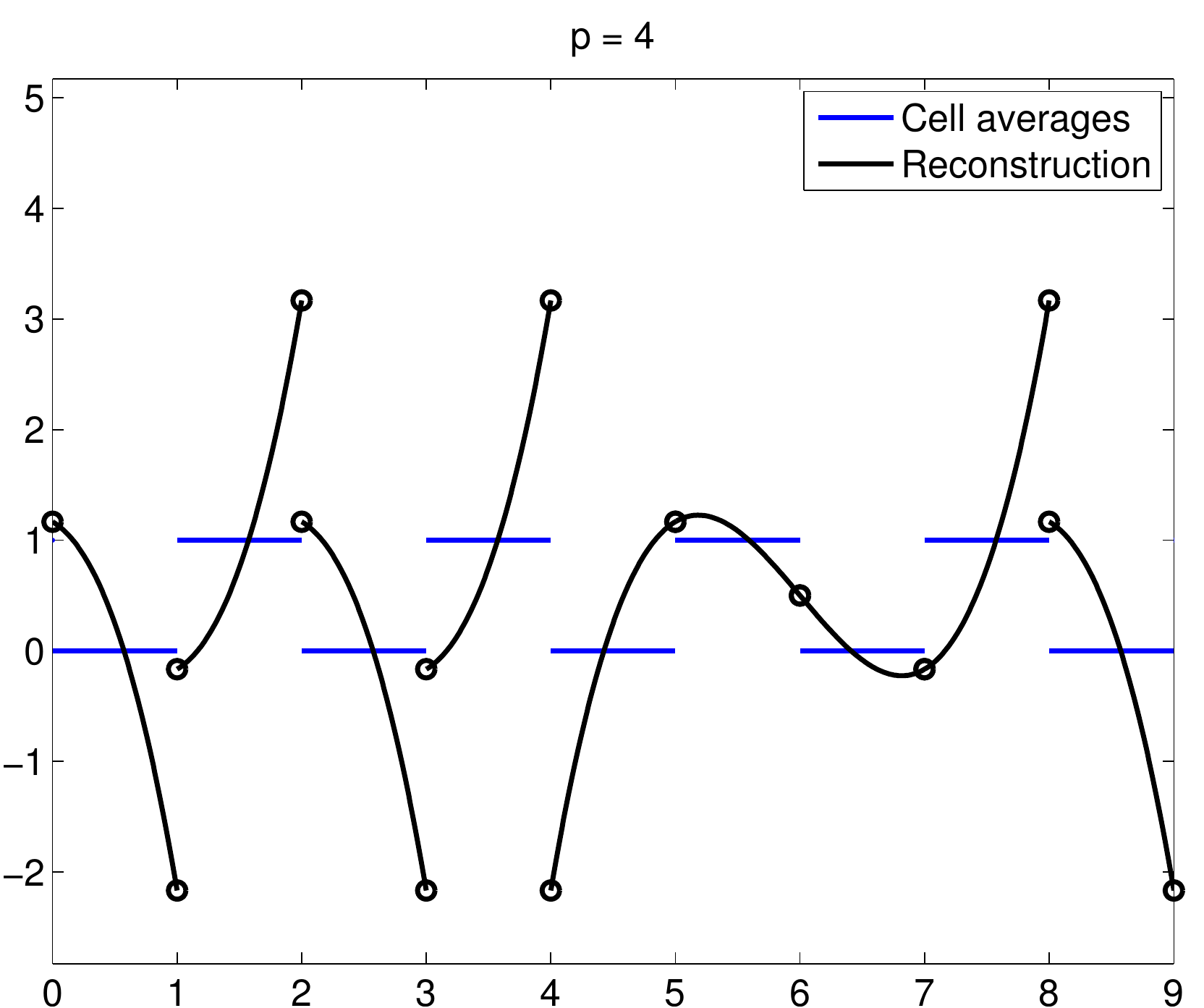}
\includegraphics[width=0.4\linewidth]{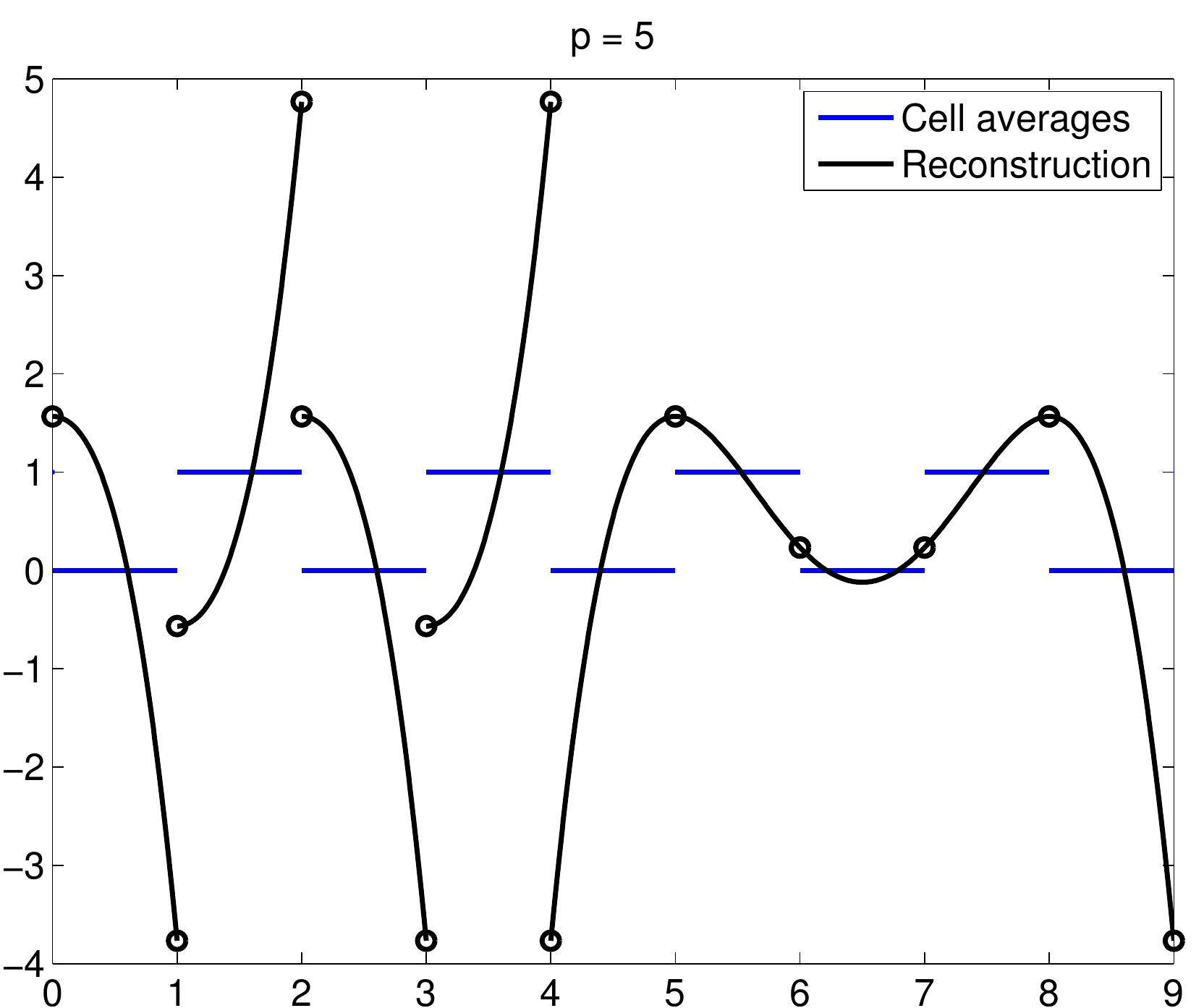}
\caption{Worst case cell interface jumps for $p=2, 3, 4, 5$.}
\label{fig:worstcase}
\end{figure}
It can be shown that the bound ${\sf C}_p$ given in Theorem \ref{thm:jmppr} is sharp. Indeed, the worst-case scenarios for orders of accuracy $p=2, 3, 4, 5$ are shown in Figure \ref{fig:worstcase}. The mesh in this figure is $x_\iphf = i$, and the cell averages are chosen as
$$
\avg{v}_i = \begin{cases} 
0 & \text{if $i$ is odd} \\
1 & \text{if $i$ is even and $i\leq 4$} \\
1-10^{-10} & \text{if $i$ is even and $i > 4$.}
\end{cases}
$$
The number $10^{-10}$ is chosen at random; any small perturbation will give the same effect. This perturbation ensures that cells $\cell_i$ for $i\leq 4$ interpolate over a stencil to the left of the cell interface $x = x_{4+\hf}$, and cells with $i>4$ to the right of it. We see that for each $p$, the jump at $x=4$ is precisely the bound ${\sf C}_p$ as given in Table \ref{tab:upperBound}.

\section{ENO interpolation}\label{sec:fdm}

The stability proof of ENO interpolation stated in Theorem \ref{thm:enosignp} can be argued along the lines of those argued in Sections \ref{sec:sign} and \ref{sec:jumpBound}. We therefore only sketch the arguments without details.

\subsection{The sign property for ENO interpolant}
We focus on the jump across the interface at $x=x_\hf$. Let $\{\rr_j\}_{j=1}^p$ and $\{\rs_j\}_{j=1}^p$ be the ENO stencils associated with gridpoints $x_0$ and, respectively, $x_1$. 
Recall that in this case of ENO interpolation, these offsets are non-positive \emph{integers},
$-p+1 \leq \rr_j,\rs_j \leq 0$.
Our first key step is to compute the jump at the interface point $x_{\hf}$ in the case where the ENO procedure are separated by just one point, namely, $\rr_p=\rs_p$. We let $\del{i}{j}$ abbreviate the divided differences $v[x_i,\dots,x_{i+j}]$. As before, we denote the reconstructed values at the interface $x=x_\hf$ by $\vminus = {\mathcal I}v(x_\hf-)$ and $\vplus={\mathcal I}v(x_\hf+)$.

\begin{lemma}
If $\rr_p = \rs_p = \rr$ for some $\rr\in-\N_0$, then
\[
\vplus-\vminus = \del{\rr}{p+1} (x_{\rr+p+1} - x_{\rr})\prod_{m=1}^{p-1}(x_\hf - x_{\rr+m}).
\]
\end{lemma}
By assembling a telescoping sum of several such stencils we obtain

\begin{corollary}
For general $\rr_p \leq \rs_p$, we have
\begin{equation}\label{eq:enoSumExprFD}
\jumplus{i+\hf} - \juminus{i+\hf} = \sum_{\rr=\rr_p}^{\rs_p} \del{\rr}{p+1} (x_{\rr+p+1} - x_{\rr})\prod_{m=1}^{p-1}(x_\hf - x_{\rr+m}).
\end{equation}
\end{corollary}

Since $v_{1}-v_0 = (x_{1}-x_0)\delii{0}{1}$, we wish to show that the jump in reconstructed values at the cell interface has the same sign as $\delii{0}{1}$.
To this end we show that each summand in \eqref{eq:enoSumExprFD} has the same sign as $\delii{0}{1}$. Indeed, since
\[
\sgn\left((x_{\rr+p} - x_{\rr})\prod_{m=1}^{p-1}(x_\hf - x_{\rr+m})\right) = (-1)^{\rr+p+1}, \qquad 
-p+1 \leq \rr \leq 0,
\]
it suffices to prove the following:
\begin{lemma}\label{lem:signPropFD}
If $\rr_p, \rs_p$ are selected according to the ENO stencil selection procedure, then
\[
\begin{cases}
\ {\rm if} \quad \delii{0}{1}\geq 0 & {\rm then} \quad (-1)^{\rr+p+1}\del{\rr}{p+1}\geq 0;\\
\ {\rm if} \quad \delii{0}{1}\leq 0 & {\rm then} \quad (-1)^{\rr+p+1}\del{\rr}{p+1}\leq 0,
\end{cases}
\quad \rr = \rr_p, \dots, \rs_p.
\]
\end{lemma}

\subsection{Upper bounds on the  relative jumps for ENO interpolant}
Next, we show the corresponding upper bound on $\vplus-\vminus$ for ENO reconstruction with point values.

\begin{lemma}
If $\rr_p, \rs_p$ are selected according to the ENO stencil selection procedure, then
\[
0\leq \frac{\del{\rr}{p+1}}{\delii{0}{1}}(-1)^{\rr+p+1} \leq \cu{r}{p} \qquad \rr = \rr_p, \dots, \rs_p,
\]
where $\cu{r}{p}$ are defined recursively, starting with $\cu{r}{1}=1$, and
\[
\cu{\rr}{p+1} = 
\frac{2}{x_{\rr+p} - x_{\rr}} \max(\cu{\rr}{p} , \cu{\rr+1}{p}).
\]
\end{lemma}

\begin{table}[ht]
\begin{center}
\begin{tabular}{|c|c|}
\hline $p$ & Upper bound ${\sf c}_p$ \\
\hline 1 & 1 \\
\hline 2 & 2 \\
\hline 3 & 3.5 \\
\hline 4 & 6 \\
\hline 5 & 10.375 \\
\hline 6 & 18.25 \\
\hline
\end{tabular}
\end{center}
\caption{}
\label{tab:upperBoundFD}
\end{table}%

For simplicity we assume that the mesh is uniform with mesh width $x_{j+1}-x_j \equiv h$. It is straightforward to show that 
$\cu{\rr}{p} \equiv \left({2}/{h}\right)^{p-1}{1}/{p!}$.
Moreover, the coefficient of the $(p+1)$-th order divided differences in \eqref{eq:enoSumExprFD} is
\[
\left|(x_{\rr+p} - x_{\rr})\prod_{m=1}^{p-1}(x_\hf - x_{\rr+m})\right| = h^{p}p\left|\prod_{m=1}^{p-1}(\hf-\rr-m)\right|.
\]
Thus, we arrive at the following bound on the jump in reconstructed values.
\begin{theorem}
Let $\rr_p, \rs_p$ be selected according to the ENO stencil selection procedure, and assume that the mesh is uniform. Then
\begin{align*}
\frac{\vplus-\vminus}{v_1-v_0} \leq {\sf c}_p := 2^{p-1}\frac{1}{(p-1)!}\sum_{\rr=0}^{p-1} \left|\prod_{m=1}^{p-1}(\hf-\rr-m)\right|.
\end{align*}
\end{theorem}

Table \ref{tab:upperBoundFD} shows the upper bound on $(\vplus-\vminus)/(v_1-v_0)$ for  $p\leq 6$. As for the ENO reconstruction procedure in Section \ref{sec:jumpBound}, it may be shown that these bounds are sharp.

\section{Conclusions}
We show that the ENO reconstruction procedure (from cell averages or point values) is stable via the \emph{sign property}, namely the jump in reconstructed values at each cell interface have the same sign as the jump in the underlying cell averages (point values). Furthermore, we obtain an upper bound on the size of the jump of reconstructed values in terms of the underlying cell averages (point values) at each interface. Both results hold for any mesh $\{x_{\iphf}\}_i$. In particular, the results hold for non-uniform meshes. In addition, both results hold for \emph{any} order of the reconstruction, i.e, any degree for the polynomial interpolation. No extra regularity assumptions on the underlying $L^1_{{\rm loc}}$ function $v$ are needed.

The proof of both the sign property and the upper jump bound depended heavily on the formula \eqref{eq:sumexpr}, which gives the cell interface jump in terms of $r_p$, $s_p$ and the $(p+1)$-th divided differences of $V$. This formula is completely independent of the ENO stencil selection procedure, and hence holds for \emph{all} interpolation stencils. On the other hand, Lemma \ref{thm:ddsign} (and Lemma \ref{thm:ddbound} for the upper bound) is a direct consequence of the ENO stencil procedure. Therefore, we cannot expect that other reconstruction methods satisfies a similar sign property. In particular, the WENO method, using the stencil weights proposed in \cite{JS96,QS02}, will in general \emph{not} satisfy such a property, a fact that is easily confirmed numerically. This leaves open the question of the existence of stencil weights that make the method satisfy the sign property. Of the second-order TVD reconstruction methods (see \cite{Swe84}), only the minmod limiter satisfies the sign property.

The stability estimates presented in this paper do not suffice to conclude that the ENO reconstruction procedure is total variation bounded (TVB). In particular, the jump in the interior of a cell can be large. However, the sign property enables us to construct arbitrarily high-order \emph{entropy stable} schemes for any system of conservation laws. Furthermore, the sign property together with the upper bound allow us to prove that these entropy stable scheme converge for linear equations. Both results are announced in \cite{FMT10} and presented in a forthcoming paper \cite{FMT11}.

\end{document}